\documentclass[10pt]{amsart}
\usepackage{lmodern}
\usepackage[T1]{fontenc}
\usepackage{microtype}

\usepackage{ucs}
\usepackage[utf8x]{inputenc}

\usepackage{amssymb,amsfonts,amsmath,mathrsfs,color,amsthm,latexsym}
\usepackage[all]{xy}
\usepackage[colorlinks=true]{hyperref}

\usepackage{graphicx}
\usepackage{rotating}
\usepackage{array}

\usepackage{float}
\usepackage{xspace}

\usepackage[english]{babel}

\usepackage{units}

\usepackage{framed}
\usepackage{stmaryrd}

\newcommand{\rank}{{\rm rank\ }}

\newtheorem{theorem}{Theorem}[section]
\newtheorem{definition}[theorem]{Definition}

\newtheorem{remark}[theorem]{Remark}
\newtheorem{proposition}[theorem]{Proposition}

\newtheorem{example}[theorem]{Example}

\begin{document}

\title{Presymplectic convexity and (ir)rational polytopes}

\author{Tudor Ratiu}
\address{School of Mathematics, Shanghai Jiao Tong 
University, 800 Dongchuan Road, Minhang District, Shanghai, 200240 China and Section de Math\'ematiques, Ecole Polytechnique F\'ed\'erale de Lausanne, CH-1015 Lausanne, Switzerland. Partially supported by NCCR SwissMAP grant of the Swiss National Science Foundation. ratiu@sjtu.edu.cn, tudor.ratiu@epfl.ch}

\author{Nguyen Tien Zung}
\address{School of Mathematics, Shanghai Jiao Tong University 
(visiting professor), 
800 Dongchuan Road, Minhang District, Shanghai, 200240 China
and Institut de Math\'ematiques de Toulouse, UMR5219, 
Universit\'e Paul Sabatier, 118 route de Narbonne,
31062 Toulouse, France, tienzung@math.univ-toulouse.fr}

\begin{abstract}
In this paper, we extend the Atiyah--Guillemin--Sternberg convexity 
theorem and Delzant's classification of symplectic toric manifolds to  
presymplectic manifolds. We also define and study the Morita 
equivalence of presymplectic toric manifolds and of their 
corresponding framed momentum polytopes, which may be rational or 
non-rational. Toric orbifolds \cite{LeTo-Orbifold1997}, quasifolds 
\cite{BaPr_SimpleNonrational2001} and non-commutative toric varieties 
\cite{KLMV-NoncommutativeToric2014} may be viewed as the quotient of 
our presymplectic toric manifolds by the kernel isotropy foliation 
of the presymplectic form. 
\end{abstract}

\maketitle

\tableofcontents

\section{Introduction}

The celebrated convexity theorem of Atiyah \cite{At1982}
and Guillemin--Sternberg \cite{GuSt-Convexity1982} states that if a 
connected compact symplectic manifold $(M,\omega)$ admits an 
effective Hamiltonian torus $\mathbb{T}^n$-action with corresponding 
momentum map $F: M \to \mathbb{R}^n$, then the image
$F(M) \subset \mathbb{R}^n$ is a convex $n$-dimensional 
polytope, called the \textit{momentum polytope}, 
which is \textit{rational}, i.e., each facet is given by a 
linear equation with rational linear coefficients in $\mathbb{R}^n$. 

A particularly important special case, related to algebraic toric 
geometry, is when the dimension of the symplectic manifold is 
exactly $2n$, where $n$ is the dimension of the torus which acts on 
it. In this case, the momentum polytope is not only rational, but 
also \textit{simple} (i.e., each $s$-dimensional face 
has exactly $n-s$ faces of dimension $s+1$ adjacent to it) and 
\textit{regular} (i.e., for any point $x_0$ on any $s$-dimensional 
face there is a complete integral affine coordinate system 
$(h_1,\hdots, h_n)$ such that the polytope is locally given near 
$x_0$ by the system of linear inequalities $\{ h_1(x) \geq 0, 
\hdots, h_{n-s}(x) \geq 0\}$). Convex polytopes which satisfy the 
three conditions of rationality, simplicity and regularity are called 
\textbf{\textit{Delzant polytopes}}, because Delzant 
\cite{Delzant1988} found a natural 1-to-1 correspondence between 
such polytopes and connected compact symplectic toric manifolds 
(i.e., those symplectic manifolds which admit an effective 
Hamiltonian torus action of half the dimension).

In this paper, we give a natural extension of these theorems to 
\textit{presymplectic} manifolds. Our main results can be roughly
formulated as follows:

$\bullet$ (Theorem \ref{thm:convexity} and Theorem 
\ref{thm:LocalSymplectization}) \textit{The image of the momentum 
map of a Hamiltonian torus action on a connected compact 
presymplectic manifold with a regular presymplectic form, under a 
natural flatness condition, is a convex polytope (of lower dimension 
in general) in a Euclidean space. Moreover, any such presymplectic 
manifold admits a unique equivariant symplectization.} 

$\bullet$ (Theorem \ref{thm:MoritaPolytopes}, Theorem 
\ref{thm:MoritaToric1} and Theorem \ref{thm:MoritaToric2}). 
\textit{Connected compact presymplectic manifolds
are classified, up to equivariant presymplectic diffeomorphisms, by 
their associated framed momentum polytopes. The classification,
up to Morita equivalence, of connected presymplectic toric 
manifolds is given by the Morita equivalence classes of their 
framed momentum polytopes.} 

Our motivation for this work comes from our desire to understand 
the role in symplectic geometry of convex simple polytopes which 
do not satisfy the rationality or the regularity conditions of 
Delzant polytopes. Many other authors have worked on this question. 
In particular, Lerman and Tolman \cite{LeTo-Orbifold1997} obtained 
the relation between non-regular simple rational polytopes and 
symplectic toric orbifolds.

Battaglia and Prato (see, e.g., 
\cite{Battaglia_ConvexPolytope2012, BaPr_SimpleNonrational2001,
BaPr_NonrationalCuts2016, BaZa_Irrational2015} and references
therein) and Katzarkov--Lupercio--Meersseman--Verjovsky (see, 
\cite{KLMV-NoncommutativeToric2014}) have worked on irrational 
analogues of symplectic toric manifolds. However, we wanted to 
have a simpler understanding of the geometric structure and so 
developed our own approach, which uses presymplectic realizations. 

In the process of studying quotient spaces of presymplectic toric 
manifolds, we are naturally led to the notion of 
\textit{Morita equivalence} of these manifolds,
and of their corresponding \textit{framed momentum polytopes}, 
borrowing the idea from the theory of Lie groupoids and stacks. 
Using our language of Morita-equivalent framed polytopes, we  
recover the results of Lerman and Tolman
\cite{LeTo-Orbifold1997} on symplectic toric orbifolds, and also 
give a clear, easy to understand, definition of what it means for 
two toric quasifolds (in the sense of Prato 
\cite{Prato_Nonrational2001}) to be isomorphic. 

We also note that, in order to turn an irrational convex polytope 
into a momentum polytope of a (pre)symplectic toric object, one 
first needs to lift it \textit{non-isomorphically} to a 
\textit{rational-faced} polytope in a higher-dimensional space! 
This simple but important observation clarifies
the role of irrational polytopes in toric (pre)sympletic geometry.

The paper is structured as follows. Section \ref{section:Convexity}
is devoted to the presymplectic version of the 
Atiyah--Guillemin-Sternberg convexity theorem. Section
\ref{section:Toric} is about presymplectic toric manifolds, 
their framed momentum polytopes, and their Morita equivalence 
classes.  Section \ref{section:Remarks}, the last section of this 
paper, contains some final remarks about related works by other 
authors and related questions. 

\section{Presymplectic convexity theorem}
\label{section:Convexity}

\subsection{The flatness condition} \hfill

The goal of this section is to prove an analogue of the 
Atiyah-Guillemin-Sternberg convexity theorem 
\cite{At1982,GuSt-Convexity1982} for Hamiltonian torus actions on 
presymplectic manifolds.

Let $(M, \omega)$ be a connected compact presymplectic manifold 
of dimension $2n+d$, i.e., $\omega \in \Omega^2(M)$ is closed. 
Suppose that the dimension of the image of the linear maps 
$T_xM \ni v_x \mapsto \omega(x)(v_x, \cdot ) \in T_x^*M$ is $2n$ for 
all $x \in M$, i.e.,
the presymplectic form $\omega$ has constant corank $d$. Assume
that there is a presymplectic torus $\mathbb{T}^{q+d}$-action on $M$
which is effective, i.e.,
the intersection of all isotropy groups of the torus action is the 
identity element. In addition, suppose that this action is Hamiltonian
in the following presymplectic sense.

Let the vector fields  $X_1,\hdots, X_{q+d}$ on $M$ 
be a family of generators of the torus action, i.e., the flow of each 
$X_i$ is periodic of period 1 and together they form the 
$\mathbb{T}^{q+d}$-action. Then, for each $i=1, \ldots, q+d$, there 
is a function $F_i: M \to \mathbb{R}$, called the
\textbf{\textit{Hamiltonian function}}
of $X_i$, such that  
$$
\omega(X_i, \cdot ) = {\rm d}F_i,
$$
similarly to the symplectic case. The map 
$$
F= (F_1, \ldots, F_{q+d}): M \rightarrow \mathbb{R}^{q+d}
$$
is called the \textbf{\textit{momentum}}
map of the Hamiltonian $\mathbb{T}^{q+d}$-action. 
It is well-known that this momentum map $F$ 
is invariant with respect to the $\mathbb{T}^{q+d}$-action, so it 
factors through a map from the space of orbits $M/\mathbb{T}^{q+d}$ 
of the $\mathbb{T}^{q+d}$-action to $\mathbb{R}^{q+d}$. We will assume 
that \textit{the kernel of $\omega$ lies in the tangent spaces to the 
orbits of the $\mathbb{T}^{q+d}$-action}. Under this assumption, 
the rank of $F$ will be at most $q$ 
at every point of $M$, and so the image 
$F(M)\subset \mathbb{R}^{q+d}$ is also of dimension at most $q$. 

In the symplectic case, when $d=0$ and $\omega$ is non-degenerate, 
the celebrated Atiayh--Guillemin--Sternberg theorem 
\cite{At1982,GuSt-Convexity1982} states that $F(M)$ is a convex 
polytope. We want to obtain a similar result for the
presymplectic case, i.e., we want to see when the image $F(M)$ of 
the momentum map $F$ of a presymplectic manifold $M$ is 
still a convex polytope. If this is the case, 
then the image (which has dimension at most $q$ by our assumptions)
must lie in a $q$-dimensional affine subspace, 
i.e., the intersection of $d$ hyperplanes in $\mathbb{R}^{q+d}$. We 
call this the \textbf{\textit{flatness condition}}
of the momentum map.

\begin{definition}
\label{def:flatness}
With the hypotheses and notations above, we say that the momentum 
map $F=(F_1, \ldots, F_{q+d}):M \rightarrow \mathbb{R}^{q+d}$ is 
\textbf{\textit{flat}} if it satisfies 
$d$ linearly independent affine relations on $M$:
\begin{equation}
\label{eqn:flatF}
\sum_{j=1}^{q+d} a_{ij}F_j = b_i, \quad a_{ij}, b_i \in \mathbb{R},
\quad i = 1,\hdots, d.
\end{equation}
In other words, the image $F(M)$ of the momentum map lies in the
$q$-dimensional intersection 
$$
L = \bigcap_{i=1}^d L_i
$$
of the hyperplanes 
\[
L_i =\left\{x=(x_1, \ldots, x_{q+d}) \in \mathbb{R}^{q+d} \, 
\left|\, \sum_{j=1}^{q+d} a_{ij}x_j = b_i\right.\right\}, \quad 
i=1, \ldots, d.
\]
\end{definition}

The above flatness condition is equivalent to the inclusions
$Y_i \in \ker \omega$ for all $i=1,\hdots, d$, where
\begin{equation}\label{eqn:Yi}
Y_i = \sum_{j=1}^{q+d} a_{ij}X_j, \quad i=1, \ldots, d,
\end{equation}
with the same constant coefficients $a_{ij}$ as in equation 
\eqref{eqn:flatF}.

\begin{example}[Flat slice] {\rm
Let $F: (\hat{M}^{2(n+d)},\omega) \to \mathbb{R}^{q+d}$ be the momentum 
map of a Hamiltonian effective torus $\mathbb{T}^{q+d}$-action on a 
connected compact symplectic manifold $(\hat{M}^{2(n+d)},\omega)$, and 
let $L$ be an arbitrary $q$-dimensional affine subspace of 
$\mathbb{R}^{q+d}$ which intersects the
$(q+d)$-dimensional polytope $F(\hat{M}^{2(n+d)})$ \textit{transversally} 
at $P = L \cap F(\hat{M}^{2(n+d)})$. Then $(M = F^{-1}(P),\omega)$ is a 
$(2n+d)$-dimensional presymplectic manifold with the inherited 
Hamiltonian torus $\mathbb{T}^{q+d}$-action from 
$(\hat{M}^{2(n+d)},\omega)$, the presymplectic form $\omega$ on $M$ has 
constant corank $d$, the inherited momentum map $F$ is flat on $M$, 
and its image $F(M) = P$ is a $q$-dimensional convex polytope. We 
say that $M$ is a \textbf{\textit{flat presymplectic slice}} of 
$(M^{2(n+d)},\omega)$ by $L$. 

If, instead of taking the transversal 
intersection of $F(\hat{M}^{2(n+d)})$ with an affine $q$-dimensional
subspace $L \subset \mathbb{R}^{q+d}$, we take its transversal 
intersection 
$$
P' = S \cap F(\hat{M}^{2(n+d)})
$$ 
with a curved $q$-dimensional submanifold $S \subset \mathbb{R}^{q+d}$, 
then $M' = F^{-1}(P')$ is still a presymplectic manifold
with a Hamiltonian torus $\mathbb{T}^{q+d}$-action on it, the kernel of 
the presymplectic form still lies in the tangent space to the orbits of the 
$\mathbb{T}^{q+d}$-action at every point, but its image under the momentum 
map is now $P'$, which is a non-convex set. 
}
\end{example}

\subsection{The presymplectic convexity theorem} \hfill

It turns out that the above flatness condition, which is of course
a necessary condition for the convexity of $F(M)$ under our assumptions, 
is also the only additional condition that one needs in order to ensure that 
$F(M)$ is a $q$-dimensional convex polytope. Moreover, if we assume that
$F(M)$ is flat $q$-dimensional, then the condition that the kernel of $\omega$
is tangent to the orbits of the torus action is automatically satisfied, 
at least at regular points of the torus action. 

\begin{theorem}\label{thm:convexity}
Let $F: M^{2n+d} \to \mathbb{R}^{q+d}$ be a flat momentum map
of a Hamiltonian torus $\mathbb{T}^{q+d}$-action on a connected 
compact presymplectic manifold $(M^{2n+d},\omega)$ whose 
presymplectic form $\omega$ has constant corank $d$. Then the image 
$F(M)$ is a convex $q$-dimensional polytope 
lying in a $q$-dimensional affine subspace $L$ of 
$\mathbb{R}^{q+d}$.
\end{theorem}

We reduce the proof of the above theorem to the symplectic case.
In order to do so, we first study the kernel of the presymplectic
form on $M^{2n+d}$. Then we show that $(M^{2n+d},\omega)$, 
together with the Hamiltonian torus action, admits a natural 
symplectization (Theorem \ref{thm:LocalSymplectization}). This 
local symplectization theorem allows us to deduce the local 
normal form of a Hamiltonian torus action in the presymplectic 
case from  the one in symplectic case, which, in turn, reduces 
the convexity problem in the presymplectic case to the well-known 
convexity result in the symplectic case.

\subsection{On the kernel of the presymplectic form} \hfill

We begin
with the following statement linking the kernel of the presymplectic
form with the tangent spaces to the $\mathbb{T}^{q+d}$-action.

\begin{proposition}
Under the assumptions of Theorem \ref{thm:convexity}, 
for every point $y \in M^{2n+d}$ we have 
$$
\ker \omega (y) \subset T_y(\mathbb{T}^{q+d}\cdot y),
$$
where $\mathbb{T}^{q+d}\cdot y$ denotes the orbit of the Hamiltonian 
$\mathbb{T}^{q+d}$-action through $y$ and $T_y(\mathbb{T}^{q+d}
\cdot y)$ is its tangent space at $y$. Moreover, the vector fields 
$Y_1,\hdots,Y_d$ given by equation \eqref{eqn:Yi}
are linearly independent  and span $\ker \omega$ at every point of 
$M^{2n+d}$.
\end{proposition}

\begin{proof}
First consider the generic case, when the $\mathbb{T}^{q+d}$-action
is locally free at $y$. Then $T_y(\mathbb{T}^{q+d} \cdot y) = 
{\rm span}(X_1,\hdots, X_{q+d}) (y)$ is of dimension $q+d$.
Its image under the contraction map $X \mapsto \omega(X,\cdot)$ 
is of dimension at most $q$, so the kernel of this linear map is 
of dimension at least $d$. However, the kernel of this linear 
map lies in $\ker \omega$, which is of dimension exactly $d$. 
Therefore, the kernel of the linear map $X \mapsto 
\omega(X,\cdot )$ on $T_y(\mathbb{T}^{q+d} \cdot y)$ coincides with 
$\ker \omega (y)$, which implies 
$\ker \omega (y) \subset T_y(\mathbb{T}^{q+d} \cdot y).$

Consider now the case when $y$ is a singular point for the torus 
$\mathbb{T}^{q+d}$-action, i.e., $\dim \mathbb{T}^{q+d} \cdot y < 
q+d$. We show that the inclusion $\ker \omega (y) \subset 
T_y(\mathbb{T}^{q+d} \cdot y)$ still holds.

By the slice theorem, there is a local submanifold $N(y)$ which 
intersects the orbit $\mathbb{T}^{q+d}\cdot y$ transversally at 
$y$ and which is saturated by the orbits of the action of 
$\mathbb{T}_y$, where $\mathbb{T}_y$ denotes the connected 
component of the identity of isotropy subgroup of the 
$\mathbb{T}^{q+d}$-action at $y$. The singular foliation by the orbits 
of the torus action is locally a direct product of the orbits of 
$\mathbb{T}_y$ on $N(y)$ with a small neighborhood of $y$ in the 
orbit $\mathbb{T}^{q+d} \cdot y$. Moreover, by local linearization, 
the orbits of $\mathbb{T}_y$ on $N(y)$
can be assumed to lie on concentric spheres centered at $y$. 

If $\ker \omega (y) \not\subset T_y(\mathbb{T}^{q+d} \cdot y)$, there 
would exist a non-zero vector $Y \in \ker \omega (y) \cap T_y N(y)$. 
By continuity, for every point $y' \in N(y)$ near $y$ which is 
regular with respect to the $\mathbb{T}^{q+d}$-action, there is also 
a vector $Y' \in \ker \omega (y')$ which is ``almost equal to $Y$''; 
$y'$ can be chosen so that $Y'$ is transverse to the cylinder which 
is the direct product of the sphere centered at $y$
in $N(y)$ with a small neighborhood of $y$ in 
$\mathbb{T}^{q+d} \cdot y$ in the local
linearized model for the torus $\mathbb{T}^{q+d}$-action. On the 
other hand, locally the orbit through $y'$ lies on this cylinder, so 
$Y' \notin T_{y'}(\mathbb{T}^{q+d} \cdot y')$, which is a 
contradiction, because $y'$ is a regular point and we must have 
$Y' \in \ker \omega (y') \subset T_{y'}(\mathbb{T}^{q+d} \cdot y')$. 

Recall that the vector fields $Y_1,\hdots, Y_d$ are tangent to 
$\ker \omega$, and are given by linearly independent linear 
combinations of $X_1,\hdots, X_{q+d}$, so at
a regular point $y$, where $X_1,\hdots, X_{q+d}$ are linearly 
independent, we also have that $Y_1,\hdots, Y_d$ are linearly 
independent and span $\ker \omega$, because 
$\dim \ker \omega (y) =d$. 

Let us show that if $y$ is a singular point, i.e.,
the connected component $\mathbb{T}_y$ of the isotropy group of the 
torus $\mathbb{T}^{n+k}$-action has positive dimension $s \geq 1$, 
then $Y_1,\hdots, Y_d$ are still linearly independent at $y$. Without 
loss of generality, we may assume that the subgroup $\mathbb{T}_y 
\subset \mathbb{T}^{n+k}$ is generated by $X_1,\hdots, X_{s}$. 
Assume that $Y_1,\hdots, Y_d$ are linearly dependent at $y$.
Then, without loss of generality, we may assume that $Y_1(y) = 0$.
In the linear combination expression
$Y_1 = \sum a_{1i}X_i$ we must have $a_{1i} = 0$ for all $i > s$ 
(because otherwise $Y_1(y)$ would be non-zero), so $Y_1 = 
\sum_{i=1}^s a_{1i}X_i$. Since the torus
$\mathbb{T}^{n+k}$-action preserves the regular integrable 
distribution $\ker \omega$, we can linearize simultaneously this 
torus action and $\ker \omega$ near $y$, i.e., find a coordinate
system $(x_1,\hdots, x_{n+d-s},z_1,\hdots,z_{n+s})$ of $M$ centered at 
$y$ in which $\ker \omega$ is a constant distribution,
$X_{s+i} = \partial/\partial x_i$ for every $i=1,\hdots, n+d-s$, and 
$X_1,\hdots, X_s$ and $Y_1$ are linear vector fields in 
$z_1,\hdots,z_{n+s}$ with imaginary eigenvalues. 
In order for $\ker \omega$ to contain $Y_1$ in this linearized coordinate 
system, $\ker \omega$ must have non-trivial intersection with the 
distribution spanned by $\partial/\partial z_1,\hdots, 
\partial/\partial z_{n+s}$ (at point $y$, 
and hence at any point in a small neighborhood of $y$ because both 
distributions are constant).
But such a nontrivial intersection is not tangent
to the vector space generated by $X_1,\hdots X_s$ at a generic point, 
and hence $\ker \omega$ will not be tangent to the 
vector space generated by $X_1,\hdots X_s, X_{s+1},\hdots, X_{n+d}$ at a 
generic point, which is a contradiction. Thus $Y_1,\hdots, Y_d$ must be 
linearly independent at $y$.
\end{proof}

\begin{proposition}
Under the assumptions of Theorem \ref{thm:convexity}, 
for every point $y \in M^{2n+d}$ we have the following 
equality:
$$
(q+d) - \dim \mathbb{T}^{n+d} \cdot y = q - 
\rank ({\rm d}F_1,\hdots, {\rm d}F_{q+d})(y) 
$$
\end{proposition}

\begin{proof}
Tbis is a direct consequence of the previous proposition: the 
map $X \mapsto \omega(X,\cdot )$ sends 
$T_y(\mathbb{T}^{n+d} \cdot y)$ to 
${\rm span} ({\rm d}F_1,\hdots, {\rm d}F_{q+d})(y)$ and its 
kernel is $\ker \omega (y) \subset T_y(\mathbb{T}^{n+d} \cdot y)$ 
which is $d$-dimensional, hence $\dim \mathbb{T}^{n+d} \cdot y = 
d + \rank ({\rm d}F_1,\hdots, {\rm d}F_{q+d})(y)$.
\end{proof}

The number $(q+d) - \dim \mathbb{T}^{n+d} \cdot y = q - 
\rank ({\rm d}F_1,\hdots, {\rm d}F_{q+d})(y)$ is called the 
\textbf{\textit{corank}} of $y$ with respect to the 
Hamiltonian torus action and is denoted by $\text{corank}\ y$. 
The point $y$ is \textbf{\textit{regular}} if and only if 
its corank is 0, or, equivalently, the orbit through $y$ has 
dimension $q+d$, or, equivalently, the momentum map has rank 
$q$ at $y$. The point $y$ is a \textbf{\textit{maximally 
singular}} if and only if the momentum map has rank zero at 
$y$, in which case the orbit through $y$ has dimension $d$ and 
is a leaf of the isotropic foliation of $\omega$.

\subsection{Local symplectization} \hfill

In order to prove Theorem \ref{thm:convexity}, we will need the following 
symplectization result.

\begin{theorem}[Local Symplectization] 
\label{thm:LocalSymplectization}
Under the hypotheses of Theorem {\rm \ref{thm:convexity}},
there exists a Hamiltonian torus $\mathbb{T}^{q+d}$-action on 
$M \times D^{d}$ (where $D^{d}\subset \mathbb{R}^d$ denotes a 
small $d$-dimensional open disk) equipped with an appropriate  
symplectic form $\tilde{\omega}$ and momentum map $\tilde{F}$ such that:
\begin{itemize}
\item [{\rm (i)}] For each $z \in D^{d}$,  $M \times \{z\}$ is 
presymplectic of constant corank $d$
and is invariant with respect to the torus 
$\mathbb{T}^{q+d}$-action.
\item [{\rm (ii)}] Denote by $O \in D^d$ the origin of the 
disk. Then $M \times \{ O \}$ together
with the pull back of $\tilde{\omega}$, the restriction of the 
$\mathbb{T}^{q+d}$-action and of $\tilde{F}$, coincides 
with the original $M$ with its presymplectic form, Hamiltonian 
$\mathbb{T}^{q+d}$-action, and momentum map $F$.
\end{itemize}

Moreover, this local symplectization is unique in the 
following natural sense. If there 
is an equivariant presymplectic embedding 
$\phi: M \to (\hat{M}^{2(n+d)},\hat{\omega})$ from 
$M$ to a symplectic manifold $(\hat{M}^{2(n+d)},\hat{\omega})$ equipped 
with a Hamiltonian  $\mathbb{T}^{q+d}$-action, then $\phi$ 
can be extended to an equivariant symplectic diffeomorphism 
from a neighborhood of $M \cong M \times \{O\}$ in $M \times 
D^d$ (equipped with the above symplectic from and Hamiltonian 
$\mathbb{T}^{q+d}$-action) into $(\hat{M}^{2(n+d)},\hat{\omega})$.
\end{theorem}

\begin{proof} Put an arbitrary $\mathbb{T}^{q+d}$-invariant 
Riemannian metric $g$ on $M$. At each point $y \in M$ denote by 
$V_y = (\ker \omega (y))^\perp \subset T_yM$ the 
$2n$-dimensional subspace of the tangent space
$T_yM$ which is $g$-orthogonal to $\ker \omega(y)$. Then the 
distribution $\mathcal{V}= \{V_y \mid y \in M\}$ is smooth 
and invariant with respect to the $\mathbb{T}^{q+d}$-action. 
For each $i=1,\hdots,d$, define the 1-form $\alpha_i$ on $M$ 
by 
\begin{equation}
\alpha_i(Y_i) = 1, \quad \alpha_i(Y_j) = 0,\; \forall j \neq i,
\quad \alpha_i (V) = 0, \forall V \text{ a section of } \mathcal{V}.
\end{equation}
Then put
\begin{equation}
\tilde{\omega} = \sum_{i=1}^d dh_i \wedge \alpha_i + 
\sum_{i=1}^n h_i d\alpha_i + proj^* \omega,
\end{equation}
where $(h_1,\hdots, h_d)$ is a coordinate system on $D^d$ 
which vanishes at $O$, and $proj: M \times D^d \to M$ is the 
natural projection onto $M$.

Lift the  $\mathbb{T}^{q+d}$-action from $M$ to $M \times D^d$ 
by making it acting trivially on $D^d$.

It is clear that $\omega$ is closed and non-degenerate (if the 
radius of the  disk $D^d$ is small) and invariant with respect 
to the  $\mathbb{T}^{q+d}$-action, which shows that  
this is action is symplectic. This symplectic action is 
actually Hamiltonian for cohomological reasons. Indeed, when 
restricted to $M \longleftarrow M \times \{O\}$, the first 
cohomology class of $\tilde{\omega}(X_i,\cdot ) = 
\omega(X_i,\cdot ) = {\rm d}F_i$ is trivial, so on 
$M \times D^d$ the cohomology class of 
$\tilde{\omega}(X_i,\cdot )$ is also 
trivial by homotopy, and hence $\tilde{\omega}(X_i,\cdot ) = 
{\rm d} \tilde{F_i}$ for some $\tilde{F_i}$ which can be 
chosen to  equal $F_i$ on $M \cong M \times \{O\}$. 

Note that, by construction, we have $X_{H_i} = Y_i$, where 
$H_i = \sum_{j=1}^q a_{ij}\tilde{F}_j$ is constant on 
$M$ for each $i=1,\hdots, d.$ 

To show uniqueness of the local symplectization, we 
invoke Gotay's coisotropic embedding theorem 
\cite{Gotay_Presymplectic1982},
or, more precisely, its equivariant version, which is proved 
using the equivariant Moser path method. We remark that, if we 
forget about the torus action, then the situation studied by Gotay 
is more general than ours, because, in his case, the isotropic 
tangent vector bundle can be non-parallelizable, while in our case 
this bundle is parallelizable (precisely because of the torus 
action).
\end{proof}

\subsection{Normal form near an orbit of the torus action} \hfill

In this subsection we recall the normal form theorem, due to 
Marle \cite{Marle_NF1984} and Guillemin--Sternberg 
\cite{GuSt1984}, for a Hamiltonian torus $\mathbb{T}^k$-action 
on a symplectic manifold in the neighborhood of an orbit of the 
action, and then adapt this theorem to our presymplectic case 
using the Symplectization Theorem 
\ref{thm:LocalSymplectization}; 
see \cite[Chapter 7]{OrRa2004} for the details and proofs of 
the well-known results stated in this subsection. 

Let us start with the following simplified (Hamiltonian 
instead of symplectic) version of the so-called Witt--Artin 
decomposition, which was first proved by Witt \cite{Witt1937} 
for symmetric bilinear forms. Fix a point $m$ in a symplectic 
manifold $M$ with a Hamiltonian $\mathbb{T}^k$-action. Since 
the torus action is Hamiltonian, every orbit is isotropic.
We split $\mathfrak{t}$, the Lie algebra of $\mathbb{T}^k$, 
into the direct sum of two summands, 
\begin{equation}
\label{splitting of lie algebras for slice} 
\mathfrak{t}= \mathfrak{t}_{m}\oplus \mathfrak{m},
\end{equation}
where $\mathfrak{m}$ is the orthogonal complement 
of $\mathfrak{t}_{m}$ in $\mathfrak{t}$ with respect to some 
positive definite inner product  $\langle\cdot,\cdot\rangle$. 
The splitting in \eqref{splitting of lie algebras for slice} 
induces a similar one on the dual 
\begin{equation}
\label{splitting of lie algebras for slice dual}
\mathfrak{t}^\ast =
\mathfrak{t}_{m}^\ast \oplus \mathfrak{m}^\ast,
\end{equation}
where $\mathfrak{t}_m^* =\{ \langle \eta, \cdot
\rangle\mid \eta \in \mathfrak{t}_m\}$ and
$\mathfrak{m}^* =\{ \langle \xi, \cdot
\rangle\mid  \xi \in \mathfrak{m}\}$.

We use the following notation. If $(V, \Omega)$ is a symplectic
vector space and $S \subset V$ is an arbitrary subset, then 
$S^\Omega = \{v\in V \mid \Omega(v, s) = 0, \,\forall x \in 
S\}$ denotes the $\Omega$-\textbf{\textit{orthogonal 
complement}} of $S$ in $V$. Note that $S^\Omega$ is a vector 
subspace of $V$.

If $G$ is a Lie group acting on a manifold $N$ whose Lie 
algebra is denoted by $\mathfrak{g}$, then the tangent space 
to the orbit $G\cdot n \subset N$ equals $\mathfrak{g}\cdot n 
= \{\xi_N(n) \mid \xi\in \mathfrak{g}\}$, where $\xi_N(n) = 
\left.\frac{d}{dt}\right|_{t=0} \exp(t \xi) \cdot n$ is the
value at $n$ of the infinitesimal generator vector field 
$\xi_N$ defined by $\xi\in \mathfrak{g}$.

\begin{theorem}[Witt--Artin decomposition]
\label{Witt--Artin decomposition statement}
Let $(M, \omega)$ be a symplectic manifold together with a 
Hamiltonian $\mathbb{T}^k$-action. Then for any point $m\in M$ 
we have
\begin{equation}
\label{witt decomposition expression}
T _m M= \mathfrak{t} \cdot m \oplus V\oplus W,
\end{equation}
where:
\begin{itemize}
\item [(i)] $V$ is the orthogonal complement to
$\mathfrak{t}\cdot m$ in $(\mathfrak{t}\cdot m) ^{\omega(m)}$
with respect to a $\mathbb{T}^k_{m}$-invariant inner product 
$\ll\! \cdot ,\! \cdot\! \gg$ in $T _mM$.
The subspace $V$ is a symplectic $\mathbb{T}^k_{m}$-invariant 
subspace of $(T_mM, \omega (m))$.

\item [(ii)] $ \mathfrak{t}\cdot m:=\{ \xi_M (m)\mid \xi\in 
\mathfrak{t}\} $ is a Lagrangian subspace of  $V^{\omega (m)}$.

\item [(iii)] $W $ is a $\mathbb{T}^k_{m}$-invariant 
Lagrangian complement to $\mathfrak{t}\cdot m $ in 
$V^{\omega (m)}$.

\item [(iv)] The map $f:W \rightarrow \mathfrak{m}^\ast$ 
defined by 
\[
\langle f (w), \eta\rangle:= \omega(m)(\eta_M (m), w), \quad 
\text{ for all } \quad \eta \in \mathfrak{m} 
\]
is a $\mathbb{T}^k_{m}$-equivariant isomorphism.
\end{itemize}
\end{theorem}

The space $V$ in Theorem 
\ref{Witt--Artin decomposition statement}
is called a \textbf{\textit{symplectic normal space}}
at $m$. The $\mathbb{T}_{m}$-action on $(V, \omega(m)|_{V})$ 
is linear Hamiltonian and has a standard associated momentum 
map $J_V:V \rightarrow \mathfrak{t}_{m}^\ast$ given 
by 
\begin{equation}
\label{V_momentum}
\langle J_V(v),\xi\rangle = \frac{1}{2}\omega(m)(\xi\cdot v,v),
\end{equation}
where $\xi \cdot v = \xi_V(v)$.

Let $(M, \omega)$ be a symplectic manifold together with a 
Hamiltonian $\mathbb{T}^k$-action and $m \in M$. Let $V$ be a 
symplectic normal space at $m$ and $\mathfrak{m} \subset 
\mathfrak{t}$ the subspace introduced in the splitting
\eqref{splitting of lie algebras for slice}.
Define the smooth manifold 
\begin{equation}
\label{marle tube canonical}
Y_r= \mathbb{T}^k \times_{\mathbb{T}^k_{m} } 
\left(\mathfrak{m}^\ast_r\times V _r\right)
\end{equation}
as the quotient of the product $\mathbb{T}^k \times 
\left(\mathfrak{m}^\ast_r\times V _r\right)$ by the 
$\mathbb{T}^k_{m}$-action defined by $h \cdot (t,\alpha, v)=
(th,\alpha, h^{-1}\cdot v)$ for any $h \in \mathbb{T}^k_{m}$, 
$t\in\mathbb{T}^k$, $\alpha\in\mathfrak{m}^\ast$, and $v\in V$.
Let $\pi:\mathbb{T}^k \times\left(\mathfrak{m}^\ast_r\times 
V _r\right) \rightarrow  \mathbb{T}^k\times_{\mathbb{T}^k_{m} } 
\left(\mathfrak{m}^\ast_r\times V _r\right)$ be the projection.
The torus $\mathbb{T}^k$ acts on $Y_r$ by the formula
$g\cdot [h, \eta, v]:=[gh, \eta,v]$, 
for any $g \in\mathbb{T}^k$ and any $[h,\eta, v] \in Y_r$.

There exist $\mathbb{T}^k_{m}$-invariant disks 
$\mathfrak{m}^\ast_r \subset \mathfrak{m}^*$ and 
$V_r \subset V$ of some small radius $r>0$\, centered at the 
origin, such that $Y_r$ is a symplectic manifold  with the 
$\mathbb{T}^k$-invariant symplectic two-form $\omega_{Y_r}$  defined by
\begin{align}
\omega_{Y _r }&([g, \rho,v])(T_{(g, \rho, v)}\pi(T _e L _g (\xi_1), 
\alpha_1, u_1),
T_{(g, \rho, v)}\pi(T _e L _g (\xi_2), \alpha_2, u _2))\notag\\
&=\langle \alpha_2+ T _v J_V (u _2), \xi_1\rangle-
\langle \alpha_1+ T _v J_V ( u _1), \xi_2\rangle +  
\omega(m)(u _1, u _2),
\label{symplectic form in the symplectic tube}
\end{align}
where $[g,\rho,v]  \in Y _r$, $\xi_1, \xi_2 \in \mathfrak{t}$, 
$\alpha_1, \alpha_2 \in\mathfrak{m}^\ast$, and $u_1, u_2 \in V$.

The symplectic manifold $(Y _r, \omega_{Y _r})$ constructed above
is called a \textbf{\textit{symplectic tube}} of $(M, \omega)$ at the 
point $m$ with respect to the Hamiltonian torus action. 
The importance of this symplectic tube is in the fact that it 
models the symplectic manifold $(M,\,\omega)$ as a Hamiltonian 
$\mathbb{T}^k$-space in a neighborhood of the orbit 
$\mathbb{T}^k\cdot m$.

\begin{theorem}[Symplectic Slice Theorem]
\label{normal form for canonical actions}
Let $(M, \omega)$ be a symplectic manifold together with a 
Hamiltonian $\mathbb{T}^k$-action. Let $m\in M$ and let 
$(Y_r, \omega_{Y_r})$ 
be the $\mathbb{T}^k$-symplectic tube at that point constructed 
above.  
Then there is a $\mathbb{T}^k$-invariant neighborhood
$U$ of $m$ in $M$ and a $\mathbb{T}^k$-equivariant 
symplectomorphism $\phi:U\rightarrow Y_r$ satisfying 
$\phi(m)=[e,\,0,\,0]$.
\end{theorem}

\begin{theorem}[The Marle--Guillemin--Sternberg normal form]
\label{thm:SymplecticNF}
Let $(M, \omega)$ be a connected symplectic manifold with a 
Hamiltonian torus $\mathbb{T}^k$-action and associated momentum 
map $F:M \rightarrow \mathbb{R}^k$.
Let $m \in M$ and $(Y_r, \omega_{Y_r})$ be the symplectic tube at 
$m$ constructed above that  models a $\mathbb{T}^k$-invariant
open neighborhood $U$ of the orbit $\mathbb{T}^k \cdot m$ via the 
$\mathbb{T}^k$-equivariant symplectomorphism $\phi:(U, \omega|_U) 
\rightarrow (Y_r, \omega_{Y_r})$. Then the momentum map 
$F_{Y _r} = F|_U \circ \phi^{-1}:Y _r \rightarrow \mathbb{R}^k$ 
of the Hamiltonian $\mathbb{T}^k$-action on $(Y_r, \omega_{Y_r})$ 
has the expression 
\begin{equation}
\begin{array}{cccc}
\label{marle normal form of the momentum}
F_{Y _r}: &Y _r=\mathbb{T}^k \times_{\mathbb{T}^k_{m}} 
(\mathfrak{m}^\ast_r \times V _r)& \longrightarrow
&\mathbb{R}^k\hfill\\
&\hfill [g, \rho, v]&\longmapsto&  F(m)+ \rho+ J_V (v),
\end{array}
\end{equation}
where $J_V: V \rightarrow \mathfrak{t}_m ^\ast $ is given by
\eqref{V_momentum}.
\end{theorem}

The above normal forms plus the symplectization theorem 
(Theorem \ref{thm:LocalSymplectization})
yields the following normal form around an orbit in the 
presymplectic case under the flatness condition for the momentum 
map.

\begin{theorem}[Presymplectic normal form under the flatness condition]
\label{thm:PresymplecticNF}
Let $(M^{2n+d}, \omega)$ be a connected presymplectic manifold, whose
presymplectic form $\omega$ has constant corank $d$, with a 
Hamiltonian torus $\mathbb{T}^{q+d}$-action and associated momentum 
map $F:M \rightarrow \mathbb{R}^{q+d}$ which satisfies the flatness 
condition (see Definition \ref{def:flatness}).
Let $m \in M$ and $(Y_r, \omega_{Y_r})$ be the symplectic tube at 
$m$ constructed above that  models a $\mathbb{T}^{q+d}$-invariant
open neighborhood $U$ of the orbit $\mathbb{T}^{q+d} \cdot m$ 
in the symplectization $(\hat{M}^{2n+2d}, \hat{\omega})$ of 
$(M^{2n+d}, \omega)$   via the 
$\mathbb{T}^{q+d}$-equivariant symplectomorphism $\phi:(U, \omega|_U) 
\rightarrow (Y_r, \omega_{Y_r})$. Then the momentum map 
$$
F_{Z_r} = F|_{U} \circ \phi^{-1}|_{Z_r}:Z_r 
\rightarrow \mathbb{R}^{q+d}
$$
of the Hamiltonian $\mathbb{T}^{q+d}$-action on 
$Z_r = Y_r \cap \phi(M^{2n+d})$ 
has the expression 
\begin{equation}
\begin{array}{cccc}
\label{presymplectic normal form of the momentum}
F_{Z_r}: &Z_r=\mathbb{T}^{q+d} \times_{\mathbb{T}^{q+d}_{m}} 
B_r& \longrightarrow
&\mathbb{R}^{q+d}\hfill\\
&\hfill [g, \rho, v]&\longmapsto&  F(m)+ \rho+ J_V (v),
\end{array}
\end{equation}
where $J_V: V \rightarrow \mathfrak{t}_m ^\ast $ is given by
\eqref{V_momentum} and 
$B_r \subset \mathfrak{m}^\ast_r \times V _r$ consists
of points $(\rho,v)$ such that $\rho + J_V (v) \in \mathfrak{l}$,
where $\mathfrak{l}$ is a $q$-dimensional linear subspace 
of $\mathfrak{t}^\ast \cong \mathbb{R}^{q+d}$  such that 
$\mathfrak{l} + \mathfrak{m}^\ast = \mathfrak{t}^\ast$.
\end{theorem}

\begin{remark}
{\rm
The space $l$ in the above theorem is nothing else but the 
vector subspace which is parallel to the $q$-dimensional 
affine subspace which contains the image $F(M)$
of the momentum map by the flatness condition. The equality
$\mathfrak{l} + \mathfrak{m}^\ast = \mathfrak{t}^\ast$ in the
theorem assures that the set $B_r$ in the theorem is a manifold,
and hence the normal form model regular. The intersection 
$\mathfrak{l} \cap \mathfrak{m}^\ast$ is not trivial in general: in fact,
it corresponds to the face of $F(M)$ which contains the point $F(m)$. 
}
\end{remark}

\subsection{Proof of convexity Theorem \ref{thm:convexity}} 
\hfill

Equipped with symplectization (Theorem \ref{thm:LocalSymplectization})
and presymplectic normal forms (Theorem \ref{thm:PresymplecticNF})
we can now obtain Theorem \ref{thm:convexity} by simply repeating the
same steps in the proofs of the classical Atiayh--Guillemin--Sternberg 
theorem.

For example, we can use the 
approach based on the local-global convexity principle, as outlined
in \cite{Zung-Proper2006}: 

Denote by $\mathcal{B}$ the $q$-dimensional 
\textbf{\textit{base space}} of a
presymplectic toric manifold $M$ with the momentum $F$, which
satisfies the conditions of Theorem \ref{thm:convexity}. By definition,
$\mathcal{B}$ is the set of all connected components of all fibers
$F^{-1}(c) \subset M$, $c \in F(M)$, endowed with the quotient 
topology, i.e., a subset $U \subset \mathcal{B}$ is open if and 
only if $\pi^{-1}\left(\tilde{U}\right)$ is open in $M$, where $\pi:
M \rightarrow \mathcal{B}$ is the map that sends $y \in M$ to the
connected component of $F^{-1}(F(y))$ containing $y$. 

The momentum map $F: M \to \mathbb{R}^{q+d}$ factorizes through 
$\mathcal{B}$, i.e., $F = \tilde{F} \circ \pi$, where $\pi$ is the
projection map from $M$ to $\mathcal{B}$ and $\tilde{F}$ is a map from
$\mathcal{B}$ to $\mathbb{R}^{q+d}$. It follows from the normal forms
(Theorem \ref{thm:PresymplecticNF}) that $\mathcal{B}$ admits a natural
intrinsic integral affine structure (see \cite{Zung-Proper2006} for a
definition of this integral affine structure using period integrals over
1-cycles), such that $\mathcal{B}$ is locally convex with locally 
polyhedral boundary with respect to this affine structure, and the 
projected momentum
map $\tilde{F}: \mathcal{B} \to \mathbb{R}^{q+d}$ is a locally injective
integral affine map. Moreover, $\mathbb{B}$ is connected compact. 
The local-global convexity principle  
(see Lemma 3.7 of \cite{Zung-Proper2006}) 
then says that $\tilde{F}$ is an injective affine
map on $\mathcal{B}$, and its image $\tilde{F}(\mathcal{B}) = F(M)$
is a convex polytope of dimension $q$ in $\mathbb{R}^{q+d}$.

Theorem \ref{thm:convexity} is proved. \hfill $\square$

\section{Presymplectic toric manifolds} 
\label{section:Toric}


\subsection{Framed momentum polytopes of presymplectic toric manifolds} 
\hfill

\begin{definition}
A compact \textit{\textbf{presymplectic toric manifold}} is a compact 
manifold $M^{2n+k}$ of dimension $2n+k$
($n,k \geq 0$) equipped with a presymplectic structure $\omega$ with 
constant corank $k$ and a Hamiltonian torus $\mathbb{T}^{n+k}$-action 
$\rho: \mathbb{T}^{n+k} \times M^{2n+k} \to M^{2n+k}$ which is free 
almost everywhere and which satisfies the flatness condition given 
in Definition \ref{def:flatness}.
\end{definition}

\begin{remark}
{\rm
Karshon and Tolman in \cite{KaTo-Presymplectic1993} also introduced
a notion of presymplectic toric manifolds, but their notion is 
completely different and should not be confused with ours. In fact, 
the presymplectic 
structure in their manifolds are symplectic almost everywhere and is 
degenerate only at a small subset, while our presymplectic structure 
is a regular presymplectic structure of constant corank. Their 
manifolds do not have convexity properties for the momentum maps, 
while our presymplectic toric manifolds
do have convex momentum polytopes. 
}
\end{remark}

It follows from Theorem \ref{thm:convexity} and Theorem 
\ref{thm:LocalSymplectization} that if $(M^{2n+k},\omega, \rho)$ is a 
compact presymplectic toric manifold with a corresponding
momentum map $F =(F_1,\hdots, F_{n+k})$, then its image $P = F(M^{2n+k})$
is an $n$-dimensional convex polytope lying in $\mathbb{R}^{n+k}$. 
Moreover, $(M^{2n+k},\omega, \rho)$ admits a unique, up to isomorphism, 
symplectization, which is an open symplectic manifold $(M^{2n+2k},\omega, 
\rho)$ together with a Hamiltonian torus $\mathbb{T}^{n+k}$-ation, and 
which contains $(M^{2n+k},\omega, \rho)$ in such a way that the 
inclusion map $i: (M^{2n+k},\omega, \rho) \hookrightarrow 
(\hat{M}^{2n+2k},\omega, \rho)$ is compatible with the (pre)symplectic 
forms, the torus actions, and the 
momentum maps, which are denoted by the same letters
for both  $M^{2n+k}$ and $\hat{M}^{2n+2k}$. 

In particular, the image $F(M^{2n+2k})$ of 
$(\hat{M}^{2n+2k},\omega, \rho)$ under the momentum map $F$ is a 
$(n+k)$-dimensional locally-polyhedral set $Q\subset \mathbb{R}^{n+k}$, 
and $P = L \cap Q$ where $L$ is an $n$-dimensional affine subspace of 
$\mathbb{R}^{n+k}$; $Q$ is like an open subset of a Delzant polytope, 
i.e., its faces satisfy the rationality, simplicity, and 
regularity conditions of a Delzant polytope. We are only interested in 
the germ of a neighborhood of $P$ in $Q$. That germ is called the 
\textit{framed momentum polytope}. To make things more precise, we 
introduce the following definitions.

\begin{definition}
{\rm (i)} A $n$-dimensional convex polytope $P \subset  \mathbb{R}^{N}$ 
($N \geq n$) is called \textbf{\textit{rational-faced}}
if for every facet $Z$ of $P$ there is a linear function with integral 
coefficients $H_Z = \sum c_j f_j$, where $c_j \in \mathbb{Z}$ and 
$(f_1,\hdots, f_{N})$ is an integral affine coordinate system on 
$\mathbb{R}^{N}$, such that $H_Z$ is constant on $Z$ but is not constant 
on $P$. 

{\rm (ii)} An injective affine map $\eta: \mathbb{R}^{N_1} \to 
\mathbb{R}^{N_2}$ ($N_1 \leq N_2)$
is called an \textit{\textbf{integral affine embedding}} from 
$\mathbb{R}^{N_1}$ to $\mathbb{R}^{N_2}$ if, up to a translation, the 
image 
$\eta(\mathbb{Z}^{N_1})$ of the integral lattice of $\mathbb{R}^{N_1}$
is a sub-lattice of the integral lattice $\mathbb{Z}^{N_2}$ of 
$\mathbb{R}^{N_1}$ such that the quotient 
$\mathbb{Z}^{N_2}/\eta(\mathbb{Z}^{N_1})$ is without torsion, or 
equivalently, the pull-back of the space  of integral affine functions on 
$\mathbb{R}^{N_2}$ to $\mathbb{R}^{N_1}$ via the map $\eta$ is exactly 
equal to the space of integral affine functions on $\mathbb{R}^{N_1}$:
$\eta^*(Af\!f_\mathbb{Z}\mathbb{R}^{N_2}) = 
Af\!f_\mathbb{Z}\mathbb{R}^{N_1}$. 
If $P \subset \mathbb{R}^{N_1}$ is a polytope and $\eta: \mathbb{R}^{N_1} 
\to \mathbb{R}^{N_2}$ is an integral affine embedding then the 
restriction of $\eta$ to $P$ is also called an \textit{\textbf{integral 
affine embedding}} from $P$ to $\mathbb{R}^{N_2}$.

{\rm (iii)} Two convex polytopes $P_1 \subset \mathbb{R}^{N_1}$ and 
$P_2 \subset \mathbb{R}^{N_2}$ ($N_1 \leq N_2$) are called 
(integral-affinely) \textbf{\textit{isomorphic}} if there is an integral 
affine embedding $\eta: \mathbb{R}^{N_1} \to \mathbb{R}^{N_2}$ such that
the restriction of $\eta$ to $P$ is a homeomorphism from $P_1$ to 
$P_2 = \eta(P_1)$. In other words, there is integral affine embedding 
from $P_1$ to $\mathbb{R}^{N_2}$ whose image is $P_2$.
\end{definition}

\begin{remark}
{\rm
It is easy to see that the above notion of integral-affinely isomorphic 
polytopes is really an equivalence relation. If two polytopes 
$P, P' \subset \mathbb{R}^N$ in the same Euclidean
space are isomorphic then it means that there is an integral affine 
transform $\phi \in GL(N,\mathbb{Z}) \ltimes \mathbb{R}^N$ such that 
$\phi(P) = P'$.
}
\end{remark}

\begin{definition} 
{\rm (i)} A \textbf{\textit{regular rational-faced framing}}  
of a convex simple polytope $P$ of dimension $n$ in $\mathbb{R}^{N}$ 
($N \geq n$) is a pair $(L,Q)$, where $L$ is the 
$n$-dimensional affine subspace of $\mathbb{R}^{N}$ which contains $P$, 
and $Q$ is a locally-polyhedral set in $\mathbb{R}^{N}$ whose faces 
satisfy the rationality, simplicity, and regularity conditions of a 
Delzant polytope, such that $L$ intersects $Q$ transversally and 
$L \cap Q = P$. The convex polytope $P$ together with a regular 
rational-faced framing given by $Q$ is called a regular 
rational-faced \textbf{\textit{framed polytope}}.

{\rm (ii)} If $P = F(M^{2n+k}) \subset L \subset \mathbb{R}^{n+k}$ 
is the image under the flat momentum map $F$
of a presymplectic toric manifold $(M^{2n+k},\omega, \rho)$ and 
$Q = F(\hat{M}^{2n+2k})$ is the image under the momentum map $F$
of a symplectization $(\hat{M}^{2n+2k},\omega, \rho)$ of 
$(M^{2n+k},\omega, \rho)$, then $P$ is called the 
\textit{\textbf{momentum polytope}} of 
$(M^{2n+k},\omega, \rho)$, and $P$ together with the framing given by 
$Q$ is called the \textit{\textbf{framed momentum polytope}} of 
$(M^{2n+k},\omega, \rho)$.

{\rm (iii)} Two $n$-dimensional regular rational-faced framed 
polytopes $P_1 = L_1 \cap Q_1$ and $P_2 = L_2 \cap Q_2$ in 
$\mathbb{R}^{n+k}$ are called (integral-affinely) 
\textbf{\textit{isomorphic}} if there is an integral affine transform
$\Phi \in GL(n+k,\mathbb{Z}) \ltimes \mathbb{R}^{n+k}$ such that 
$\Phi(L_1) = L_2$ and $\Phi(U(P_1)) = U(P_2)$ where $U(P_1)$ (resp., 
$U(P_2)$) is a small neighborhood of $P_1$ (resp., $P_2$) in $Q_1$ 
(resp., $Q_2$).
\end{definition}

\begin{theorem}
{\rm (i)} If $(M^{2n+k},\omega, \rho)$ is a connected 
compact presymplectic toric manifold, 
then its momentum polytope is a convex rational-faced simple 
polytope and its framed momentum polytope is a regular rational-faced 
framed polytope.

{\rm (ii)} Conversely, any convex regular rational-faced framed 
polytope is the  framed momentum polytope of a compact presymplectic 
toric manifold.

{\rm (iii)} Connected compact presymplectic toric manifolds are 
classified by their framed momentum polytopes: 
two compact presymplectic toric manifolds are 
isomorphic if and only if their corresponding framed momentum polytopes 
are isomorphic.
\end{theorem}

\begin{proof} 

Part {\rm (i)} is just a special case of the results obtained in Section
\ref{section:Convexity}. The image $Q = F(M^{2n+k})$ of a
symplectization $(\hat{M}^{2n+2k},\omega, \rho)$ of our presymplectic 
manifold is locally-polyhedral and satisfies the rationality, simplicity 
and regularity properties at its face because the singularities of the 
Hamiltonian torus action, viewed as a toric integrable Hamiltonian system 
on it, are all non-degenerate elliptic. The momentum polytope $P$ of 
$(M^{2n+k},\omega, \rho)$ is a convex polytope because of the convexity 
Theorem \ref{thm:convexity}, and this polytope is simple because it is a 
slice of $Q$, which is simple, by a affine submanifold which cuts $Q$ 
transversally (the transversality condition is implied by the regularity
condition of the presymplectic form $\omega$ on $M^{2n+k}$). 

For Part {\rm (ii)}, given a regular framing $(P,Q)$ of $P$, there is a 
unique integrable Hamiltonian system with elliptic singularities which 
admits $Q$ (together with its induced integral affine structure) as the 
base space, according to the general results of Zung 
\cite{Zung-Integrable2003} on the construction and classification of 
integrable Hamiltonian systems. Due to the type of the base space and 
singularities, this integrable Hamiltonian system is actually a 
Hamiltonian torus action of half the dimension of the symplectic manifold 
$(M^{2n+2k},\omega)$. By taking $M^{2n+ k} = F^{-1}(P)$ with the 
pull-back of the symplectic form and the restriction of the torus 
action on it, where $F: M^{2n+2k} \to \mathbb{R}^{n+k}$ is the momentum 
map of the Hamiltonian torus $\mathbb{T}^{n+k}$-action such that 
$F(M^{2n+2k}) = Q$, we get the required presymplectic toric manifold whose
framed momentum polytope is $(P,Q)$. Part {\rm (iii)} also follows from 
these same arguments, together with Theorem \ref{thm:LocalSymplectization}
about the existence and uniqueness of equivariant symplectization.
\end{proof}

\begin{remark}[Slices of Delzant polytopes]
{\rm  When $Q$ is a framing of $P$ then we also say that $P$ is a 
\textbf{\textit{slice}} of $Q$. If $Q$ is a Delzant polytope and 
$P = L \cap Q$ is a slice of $Q$ by an affine subspace which intersects 
$Q$ transversally, then, of course, by Delzant's theorem $Q = F(M)$ is 
the image of the momentum map $F$ of a toric manifold $M$ and 
$M_P = F^{-1}(P)$ is the presymplectic toric submanifold of $M$ whose 
momentum polytope is $P$. In general, we do not need $Q$ to be a 
regular simple polytope, we just need that locally $Q$ looks like a 
regular simple polytope at its faces. 
}
\end{remark}

\subsection{Lifting and framing of polytopes} \hfill
\label{subsection:LiftingFraming}

In the definition of a rational-faced polytope $P$,  the
affine subspace $L \subset \mathbb{R}^N$ containing $P$ may be 
irrational, in the sense that the linear equations defining it may have 
irrational linear coefficients, even though each facet of $P$
must satisfy a rational linear equation. If $L$ is rational, then we 
also say that $P$ is \textbf{\textit{rational}}, and if $L$ is 
irrational then we also say that $P$ is \textbf{\textit{irrational}}.

More precisely, we can define the \textbf{\textit{degree of 
irrationality}} to be the minimal number of linear equations which 
must have at least one irrational linear coefficient in the definition 
of the supporting affine subspace $L$ of $P$. It is clear that $P$ 
is rational if and only if its degree of irrationality is 0, and if 
$P$ and $P'$ are isomorphic  then they have the same degree of 
irrationality.

For each convex rational-faced polytope $P \subset \mathbb{R}^N$, 
denote by $Af\!f_\mathbb{Z}(P)$ the Abelian group of all integral 
affine functions restricted to $P$, i.e., the quotient of the space 
of all integral affine functions on  $\mathbb{R}^N$ by those which 
vanish on $P$. The rational-faced condition means that each facet 
of $P$ is given by an equation of the form $F = 0$, where 
$F \in Af\!f_\mathbb{Z}(P)$. The quotient $Da\!f\!f_\mathbb{Z}(P) = 
Af\!f_\mathbb{Z}(P)/\mathbb{R}$ of $Af\!f_\mathbb{Z}(P)$ by constant 
functions is a free finitely generated Abelian group, and is called 
the \textbf{\textit{Abelian group of integral affine 1-forms}} on $P$.

Up to isomorphisms, each convex rational-faced polytope $P$ is uniquely
characterized by its group of integral affine functions 
$Af\!f_\mathbb{Z}(P)$, i.e., $P_1$ is isomorphic to $P_2$ (even if they 
live in different Euclidean spaces) if and only if there is a 
homeomorphism from $P_1$ to $P_2$ which induces a group isomorphism from
$Af\!f_\mathbb{Z}(P_1)$ to $Af\!f_\mathbb{Z}(P_2)$. The number 
$$ 
d = {\rm rank}_\mathbb{Z} Da\!f\!f_\mathbb{Z}(P) - \dim P
$$ 
is nothing else but the \textit{degree of irrationality} of $P$. In
addition, $P$ is a subset of a Euclidean space of dimension at least 
$n+d$, where $n$ is the dimension of $P$ and $d$ is
the degree of irrationality of $P$. 

One can embed $P$ isomorphically into $\mathbb{R}^{n+d}$ by a map 
$G = (G_1,\hdots,G_{n+d}): P \to \mathbb{R}^{n+d}$, where 
$(G_1,\hdots,G_{n+d})$ modulo constants is a basis of 
$Da\!f\!f_\mathbb{Z}(P)$. However, in order to find a 
\textit{regular} simple 
rational-faced framing for $P$, we may need more than $n+d$ dimensions. 
Indeed, already if $P$ is $n$-dimensional simple rational but not 
regular in $\mathbb{R}^n$, we need more dimensions in order to produce 
a regular framing of $P$. 

By increasing the dimension of the Euclidean space, if necessary, it is 
always possible to find a regular rational-faced framing for a
rational-faced simple convex polytope, as will be shown in Theorem 
\ref{thm:LiftingFraming}.

By an \textit{irrational polytope}, many authors, including 
Prato--Battaglia \cite{BaPr_SimpleNonrational2001} and 
Katzarkov--Lupercio--Meersseman--Verjovsky
\cite{KLMV-NoncommutativeToric2014}, mean an $n$-dimensional polytope 
in $\mathbb{R}^n$ with an irrational face. However, the rational-faced 
property of a polytope is preserved under isomorphisms, 
so if a polytope $P$ has an irrational face, there is no way to turn 
it into the momentum polytope of a presymplectic toric manifold by 
integral affine isomorphisms.

Nevertheless, one can always \textit{lift} an arbitrary convex polytope
$P \subset \mathbb{R}^n$ to a rational-faced polytope $P' \subset 
\mathbb{R}^N$ with a regular rational-faced framing $Q \subset 
\mathbb{R}^N$ for some $N > n$, such that the linear  projection map 
$proj:\mathbb{R}^N \to \mathbb{R}^n$ on the first $n$ components of 
$\mathbb{R}^N$ projects $P'$ to $P$ homeomorphically. If $P$ is 
rational-faced then projection map $proj$ is an integral-affine 
isomorphism from $P'$ to $P$, and if $P$ is not rational-faced, then it 
is not. We call such a $P'$ a \textbf{\textit{rational-faced lifting}} 
of $P$.

We describe below an easy construction of rational-faced liftings 
together with a regular framing, which is already used by Prato in 
\cite{Prato_Nonrational2001}.

Let $P \subset \mathbb{R}^n$ be a arbitrary convex polytope of 
dimension $n$. Let $F_i(x) = \sum_{j=1}^n a_{ij}x_j + b_i$, 
where $a_{ij}, b_i \in \mathbb{R}$ and $(x_1,\hdots,x_n)$ is an 
integral affine coordinate system of $\mathbb{R}^n$, be linear 
functions on $\mathbb{R}^n$ which determine the facets of $P$ 
($i = 1,\hdots,k$), and such that $F_i \geq 0$ on $P$. Then $P$ 
can be written as 
$$
P=\{x \in \mathbb{R}^n \mid F_i(x) \geq 0,\, i=1,\ldots, k\}.
$$ 
Denote the coordinates 
of points in $\mathbb{R}^{n+k}$ by $(x_1,\hdots, x_n, y_1,\hdots, y_k)$.
For $M>0$ sufficiently large, define the box
\begin{align*}
Q = \{(x_1,\hdots, x_n, y_1,\hdots, y_k) \in\mathbb{R}^{n+k} \mid&
-M \leq x_i \leq M,\,i=1, \ldots, n,\\
&\;\;0 \leq y_j \leq M, \, j= 1, \ldots, k\}.
\end{align*}
Cut $Q$ by the $k$ hyperplanes $L_i$, $i = 1,\hdots, k$, where 
$$
L_i = \left\{(x_1,\hdots, x_n, y_1,\hdots, y_k) \in\mathbb{R}^{n+k}\,
\left|\,
y_i = \sum_{j=1}^n a_{ij}x_j + b_i \right. \right\}
$$
to obtain $P' = Q \cap L_1 \cap \hdots \cap L_k$.  
It is easy to see that if $M$ is large enough, then
$P' \subset \mathbb{R}^{n+k}$ projects bijectively 
to $P \subset \mathbb{R}^n$, and that $Q$ is a Delzant polytope, 
hence a regular framing for $P'$.

When $P$ is rational-faced, we can choose all the above coefficients 
$a_{ij}$ to be integers. In this case, because the coefficient of 
$y_i$ in the equation $y_i = \sum_{j=1}^n a_{ij}x_j + b_i$
is 1 for every $1\leq i \leq k$, it is easy to verify that the 
projection from $P'$ to $P$ is an integral affine isomorphism. 

If the dimension of $P \subset \mathbb{R}^n$ is smaller than  
$n$, the situation is the same: just add the same defining linear 
equations for $P$ to the above linear equations $y_i = 
\sum_{j=1}^n a_{ij}x_j + b_i$ for $P'$; $Q$ remains the same. From 
this construction, we obtain the following result.

\begin{theorem} \label{thm:LiftingFraming}
For any convex simple polytope $P\subset \mathbb{R}^N$ there is 
another polytope $P'\subset \mathbb{R}^{N'}$ ($N \geq N'$) which 
projects to $P$ under the natural projection from $\mathbb{R}^{N'}$ 
to $\mathbb{R}^{N}$, such that $P'$ admits a regular rational-faced 
framing in $\mathbb{R}^{N'}$, and hence is the momentum polytope of 
a presymplectic toric manifold. Moreover, if $P$ is rational-faced 
the $P'$ is integral-affinely isomorphic to $P$.
\end{theorem}

\begin{remark} {\rm
Even when $P$ is not simple, we can use the same construction.
The only difference is that $L$ does not intersect $Q$ transversally, 
if $P$ is not simple.
}
\end{remark}

\begin{example}
{\rm 
In general, the same convex rational-faced polytope  admits 
infinitely many regular framings which are not isomorphic, and thus 
correspond to infinitely many non-isomorphic presymplectic toric 
manifolds. Take, for example, the interval 
$P = [O,A] \subset \mathbb{R}^2$, where $O=(0,0)$ and $A = (a,0)$. 
Then $P$ is rational-faced and admits infinitely many non-isomorphic 
2-dimensional regular rational-faced framings. In fact, for any 
positive integer $p$ and any integer $q$ such that $gcd(p,q) =1$, 
the set 
$$
Q_{p,q} = \{(x,y) \in \mathbb{R}^2 \mid 
-\varepsilon < y < \varepsilon, \;
x \geq 0, \; p(x-a) + qy \leq 0 \}
$$
(for $\varepsilon > 0$ small enough) is a regular rational-faced 
framing of $P$, and two framings $Q_{p,q}$ and $Q_{p',q'}$ are 
isomorphic if and only if $p = p'$ and $q = \pm q'$.
}
\end{example}

\subsection{Morita equivalence} \hfill

The notion of Morita equivalence that we want to introduce in this
subsection is inspired by the notion of Morita equivalence for Lie 
groupoids (see, e.g., \cite{Mackenzie2005}, 
\cite[Section 7.2]{DuZu2005}).
Intuitively speaking, two presymplectic toric manifolds 
$(M_1,\omega_1, \rho_1)$ and $(M_2,\omega_2, \rho_2)$ are Morita
equivalent if their quotient spaces with respect
to the corresponding kernel isotropic foliations are isomorphic. 

In the case of rational momentum polytopes, the quotient spaces of 
the presymplectic toric manifolds are symplectic toric orbifolds and 
we can really compare them directly. However, when the momentum 
polytopes are irrational, the quotient spaces are quasifolds which
are not Hausdroff, and it is rather inconvenient to compare such 
bad quotient spaces directly. Instead, we will develop the Morita 
equivalence as an indirect way to verify when two presymplectic 
toric manifolds should be considered as having the same quotient space.

\begin{definition}
\label{defn:MoritaPresymplectic}
{\rm (i)} Let $\phi: (M_1^{2n+k_1},\omega_1, \rho_1) \to 
(M_2^{2n+k_2},\omega_2, \rho_2)$ be a submersion 
with connected fibers between two presymplectic toric manifolds $M_1$ 
and $M_2$, with $k_1 \geq k_2$ ($k_i$ is the corank of $\omega_i$).
Then $\phi$ is called a \textbf{\textit{Morita equivalence submersion}}  
if $\omega_1 = \phi^* \omega_2$, and 
$\phi(\rho_1(t,x)) = \rho_2(\theta(t),\phi(x))$ for any $x \in 
M_1^{2n+k_1}$ and any $t \in \mathbb{T}^{n+k_1}$,
where $\theta: \mathbb{T}^{n+k_1} \to \mathbb{T}^{n+k_2}$ is a 
surjective homomorphism whose kernel is connected.

{\rm (ii)} Two presymplectic toric manifolds 
$(M_1^{2n+k_1},\omega_1, \rho_1)$ and $(M_2^{2n+k_2},\omega_2, 
\rho_2)$ are called \textbf{\textit{Morita equivalent}} if there 
is a third presymplectic toric manifold $(M_3^{2n+k_3},\omega_3, 
\rho_3)$ together with two Morita equivalence submersions 
$\phi_1: M_3^{2n+k_3} \to M_1^{2n+k_1}$ and 
$\phi_2: M_3^{2n+k_3} \to M_2^{2n+k_2}$.
\end{definition}

It is not obvious from the definition that the above Morita 
equivalence notion is a true equivalence relation
among presymplectic toric manifolds, but it really is. In order to 
see why, we can translate this equivalence relation to an 
equivalence relation among framed momentum polytopes
of presymplectic toric manifolds.

\begin{definition}\label{defn:MoritaPolytopes}
{\rm (i)} Let $Q_1 \subset \mathbb{R}^{N_1}$ and $Q_2 \subset 
\mathbb{R}^{N_2}$ be two regular rational-faced 
framings of two rational-faced simple polytopes $P_1 \subset 
\mathbb{R}^{N_1}$ and $P_2 \subset \mathbb{R}^{N_2}$ respectively. 
Assume that there is an integral affine embedding 
$\eta: \mathbb{R}^{N_2} \to\mathbb{R}^{N_1} $ from 
$\mathbb{R}^{N_2}$ to $\mathbb{R}^{N_1}$ ($N_1 \geq N_2$), such 
that $\eta(P_2) = P_1$ and $\eta(U(P_2)) = U(P_1)$ where 
$U(P_i)$ is a small neighborhood of $P_i$ in $Q_i$ respectively
($i=1,2$). Then we say that the framed polytope $(P_1,Q_1)$ is 
\textbf{\textit{Morita-equivalent}} to the framed polytope 
$(P_2, Q_2)$, and that $\eta$ is a \textbf{\textit{Morita 
equivalence embedding}} from $(P_2, Q_2)$ to $(P_1,Q_1)$. 

{\rm (ii)} Two framed polytopes are called 
\textbf{\textit{Morita-equivalent}} if both of them admit 
Morita equivalence embeddings to a third framed polytope.
\end{definition}

\begin{theorem} \label{thm:MoritaPolytopes}
The Morita equivalence of regular rational-faced framed simple convex 
polytopes is a true equivalence relation.
\end{theorem}

\begin{proof} All framings in this proof are assumed to be 
regular simple.
It is easy to see directly from Definition \ref{defn:MoritaPolytopes} 
that if $\eta$ is a Morita equivalence embedding from a framed polytope
$(P_1,Q_1)$ to a framed polytope $(P_2,Q_2)$, and $\nu$
is a Morita equivalence embedding from
$(P_2,Q_2)$ to a framed polytope $(P_3,Q_3)$, then the composition
$\nu \circ \eta$ is a Morita equivalence embedding from
$(P_1,Q_1)$ to $(P_1,Q_1)$.

The main point in the proof of the above theorem is the verification 
of the following statement: if $(P,Q)$ admit two Morita equivalence 
embeddings to $(P_1,Q_1)$ and $(P_2,Q_2)$ then there exist Morita 
equivalence embeddings from $(P_1,Q_1)$ and $(P_2,Q_2)$ to another 
framed polytope $(P_3,Q_3)$.

We will construct $(P_3,Q_3)$ as a ``\textit{crossed product of 
$(P_1,Q_1)$ and $(P_2,Q_2)$ over $(P,Q)$}''. The construction goes 
as follows: for each $i = 1,2$,
decompose the ambient Euclidean space $V_i \cong \mathbb{R}^{N_i}$ 
of $Q_i$ in an integral affine way 
as $V_i = K_i \oplus V$ (the integral affine structure on $V_i$ is 
the direct sum of
the integral affine structures on $K_i$ and $V$), where 
$V \cong \mathbb{R}^N$
is identified with the ambient Euclidean space of $Q$ via the Morita 
equivalence embedding from $(P,Q)$ to $(P_i, Q_i)$. $P$ is 
identified with $P_1$ and $P_2$ via these
embeddings. Put $V_3 = K_1 \oplus K_2 \oplus V \cong 
\mathbb{R}^{N_1+N_2-N}$, which contains both $V_1$ and $V_2$ via 
natural identifications, and define the framing $Q_3$
of $P_3 = P$ in $V_3$ as follows:

Each facet $\zeta$ of $Q$ is contained in exactly one facet 
$\zeta_1$ of $Q_1$ and exactly one facet $\zeta_2$ of $Q_2$ in $V$, 
$\zeta_1 \cap \zeta_2 = \zeta.$ The (smallest) affine 
subspace of $V_3$ which contains both $\zeta_1$ and $\zeta_2$ is 
a hyperplane (i.e., of codimension 1) in $V_3$. Denote by 
$\tilde{\zeta}_3$ the half-space of $V_3$ bounded
by this hyperplane which contains $P$. Take $\tilde{Q}_3$ to be 
the intersection of all these half-spaces (one for each facet of $Q$), 
and define $Q_3$ to be a small neighborhood of $P$ in this 
intersection. This is the framing of $P$ that we wanted to construct. 

It is clear from the construction that $Q_3$ is simple, rational, 
and that both $(P,Q_1)$ and $(P,Q_2)$ are embedded in $(P,Q_3)$ 
in an integral affine way. It remains to check that $Q_3$ is regular, 
but this fact is a consequence of our assumption that 
the three frames $Q, Q_1, Q_2$ are all regular. 

Indeed, consider a face of $Q_3$ and prove that $Q_3$ is regular 
at that face. It's enough to show that $Q_3$ is regular at one point 
$x$ of that face, and we can choose $x$ to be in $Q$, because $Q$ is 
a transversal slice of $Q_3$ and each face of $Q_3$ contains a face 
of $Q$. 

Denote by $\zeta^1,\hdots, \zeta^m$ the facets of $Q$ which contain 
$x$ ($m \geq 1$). The regularity of $Q$ at $x$ means that on the 
tangent space $T_xQ$, equipped with the integral lattice induced
from $V \supset Q$, there is a basis $(\alpha_1,\hdots,\alpha_N)$ 
of this integral lattice such that
$\alpha_i \in T_x(\zeta^1 \cap \hdots \cap \zeta^{i-1} \cap \zeta^{i+1} 
\cap \hdots \cap \zeta^m)$ for each $i= 1,\hdots,m$ and 
$(\zeta_{m+1},\hdots,\zeta_N)$ is a basis for the integral lattice of
$T_x(\zeta^1 \cap \hdots \cap \zeta^m)$. 

Because $Q_1$ is also regular at $x$, we have a similar basis for 
the integral lattice of $T_xQ_1$.
Actually, because of the embedding of $Q$ in $Q_1$, we can choose 
the basis of the integral lattice of $T_xQ_1$ to be of the form 
$(\alpha_1,\hdots,\alpha_N,\beta_1,\hdots,\beta_{N_1-N})$, where
$(\alpha_1,\hdots,\alpha_N)$ is the above basis for $T_xQ$ and 
$\beta_1,\hdots,\beta_{N_1-N}$ are additional vectors in the tangent 
space to the intersection of the facets of $Q_1$ at $x$.
For the same reasons, we have a similar basis 
$(\alpha_1,\hdots,\alpha_N,\gamma_1,\hdots,\gamma_{N_2-N})$ for 
the integral affine lattice of $T_xQ_2$. Then 
$(\alpha_1,\hdots,\alpha_N,\beta_1,\hdots,\beta_{N_1-N}, \gamma_1,
\hdots, \gamma_{N_2-N})$ is a basis for the integral lattice of 
$T_xQ_3$, which  implies that $Q_3$ is regular at $x$.
\end{proof}

\begin{theorem} \label{thm:MoritaToric1}
There  is a Morita equivalence submersion 
$\phi: (M_1^{2n+k_1},\omega_1, \rho_1) \to (M_2^{2n+k_2},\omega_2, 
\rho_2)$ between two presymplectic toric manifolds if and only if 
there is a Morita equivalence embedding $\eta$ from the corresponding 
second framed momentum polytope $(P_2, Q_2)$ to the first framed 
momentum polytope $(P_1, Q_1)$
\end{theorem}

\begin{proof} Let $\phi: (M_1^{2n+k_1},\omega_1, \rho_1) \to 
(M_2^{2n+k_2},\omega_2, \rho_2)$ be a Morita equivalence submersion 
between two presymplectic toric manifolds. By definition,
we have that $\omega_1 = \phi^* \omega_2$, hence the tangent spaces 
to the fibers of the projection map $\phi: M_1 \to M_2$ lie in the 
kernel of $\omega_1$. These tangent spaces have dimension 
$s = k_1 - k_2$. 

Denote by $\mathbb{T}^s = \theta^{-1}(0)$ the kernel of $\theta$ 
(it is a torus of dimension $s$), where $\theta: \mathbb{T}^{n+k_1} 
\to \mathbb{T}^{n+k_2}$ is the surjective homomorphism given in 
Definition \ref{defn:MoritaPresymplectic} of Morita equivalence 
submersion. For any $x \in M_1$ and $t \in \mathbb{T}^s$ we have 
$\phi (\rho_1(t,x)) = \rho_2(\theta(t), \phi(x)) =
\rho_2(0, \phi(x)) = \phi(x)$, which means that the orbit of the 
action of $\mathbb{T}^s$ through $x$ in $M_1$ lies in the fiber 
of $\phi$ which contains $x$. 

Consider a symplectization $(\hat{M}_1^{2n+2k_1},\omega_1, \rho_1)$ of 
$(M_1^{2n+k_1},\omega_1, \rho_1)$ and the sub-Hamiltonian 
$\mathbb{T}^s$-action 
of the Hamiltonian $\mathbb{T}^{n+k_1}$-action on it. Denote by
$(H_1,\hdots,H_s): \hat{M}_1^{2n+2k_1} \to \mathbb{R}^s$ the 
momentum map of this $\mathbb{T}^s$-action, and by $X_1,\hdots,X_s$
the corresponding infinitesimal generators,
with $X_i$ being the Hamiltonian vector field of $H_i$. 
Since $X_i$ lies in the kernel
of $\omega_1$ on $M_1$, it follows that $H_i$ is constant on $M_1$ for 
every $i=1,\hdots, s$, and without losing generality we can assume 
that $H_i = 0$ on $M_1$.

Since $X_i$ are in the kernel of $\omega_1$ on $M_1$, the results 
of the previous section say that the vector fields $X_1,\hdots,X_s$ 
are independent everywhere on $M_1$, i.e., the action of 
$\mathbb{T}^s$ on $M_1$ is locally free everywhere. It follows that
the orbits of $\mathbb{T}^s$ on $M_1$ have dimension $s$ everywhere, 
and hence  coincide with the fibers of the submersion $\phi$. The 
action of $\mathbb{T}^s$ is free almost everywhere (being a
sub-action of a $\mathbb{T}^{n+k_1}$-action which is free
almost everywhere), and its orbits form a locally trivial 
fibration, so the action must actually be free everywhere (without 
any discrete isotropy at any point).

Consider the Marsden-Weinstein reduction of 
$(\hat{M}_1^{2n+2k_1},\omega_1)$ with respect
to the $\mathbb{T}^s$-action at the zero level 
$H_1=\hdots = H_s =0$. One verifies easily that
the reduced symplectic manifold is a symplectization of 
$(M_2^{2n+k_2},\omega_2, \rho_2)$
and the image of the momentum map of the Hamiltonian 
$\mathbb{T}^{n+k_2}$-action on it coincides with 
$Q' = Q_1 \cap \{H_1 = 0\} \cap \hdots \cap \{H_s = 0\}$, 
which is a regular framing of $P$ in an $(n+k_2)$-dimensional 
space. Due to the uniqueness of symplectization
of $(M_2^{2n+k_2},\omega_2, \rho_2)$ up to isomorphisms, we have 
that the framed polytope $(P_2,Q_2)$ is isomorphic to $(P,Q')$, 
which means that there is a Morita equivalence
embedding from $(P_2,Q_2)$ to $(P_1,Q_1)$. 

The converse statement can be proved in a similar way. 
\end{proof}

\begin{theorem}\label{thm:MoritaToric2}
The Morita equivalence of compact presymplectic toric manifolds 
is a true equivalence relation, and two compact presymplectic 
toric manifolds are Morita-equivalent if
and only if their corresponding framed momentum polytopes are
Morita-equivalent.
\end{theorem}

\begin{proof} It is a direct consequence of Theorem 
\ref{thm:MoritaPolytopes}
and Theorem \ref{thm:MoritaToric1}.
\end{proof}

\begin{example} {\rm 
The two intervals $P_1 =[A,B]$ and $P_2 = [G,H]$ in Figure 
\ref{fig:Framing1} are isomorphic. Their framings, also shown in 
Figure \ref{fig:Framing1} are not isomorphic, but are Morita-
equivalent. The framed $P_2$ corresponds to the presymplectic
manifold $S^3$ with the Hopf circle fibration as the  kernel 
isotropic foliation, while the framed $P_1$ corresponds to the 
presymplectic manifold $S^2 \times S^1$ whose kernel isotropic 
foliation is the projection to $S^2$. Both have $S^2$ as the
quotient space. By taking a direct product of the framing of $P_2$ 
with an interval, one can easily realize a 3-dimensional framing of 
$P_2$ which is Morita-equivalent to both the framed $P_1$ and
the framed $P_2$ via integral affine embeddings. 
}
\end{example}

\begin{figure}[!ht]
\centering
\includegraphics[width=\textwidth]{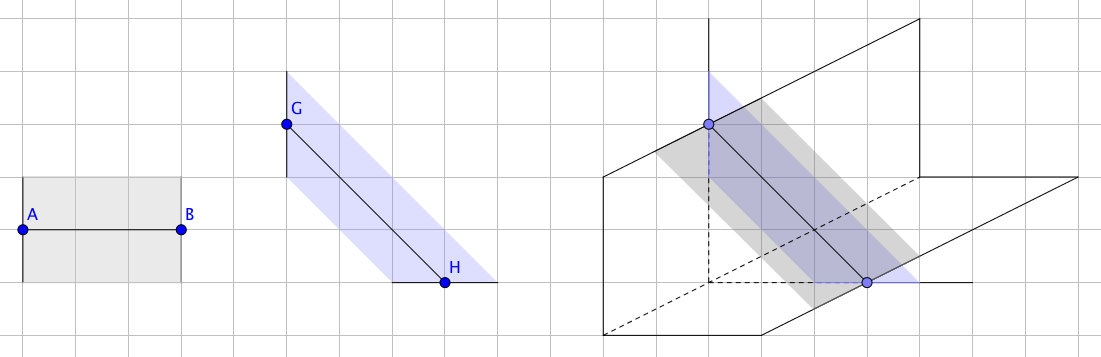}
\caption{Non-isomorphic but Morita-equivalent framed intervals}
\label{fig:Framing1}
\end{figure}

\subsection{Toric orbifolds and quasifolds} \hfill 

Theorem \ref{thm:MoritaToric2} reduces the problem of classification
of presymplectic toric manifolds up to Morita equivalence to the 
combinatorial problem of classification of framed polytopes up to
Morita equivalence. In particular, if two presymplectic toric 
manifolds are Morita equivalent then their momentum polytopes must
be integral-affinely isomorphic, though the same polytope (without
framing) can correspond to infinitely many non-Morita-equivalent 
presymplectic toric manifolds.

It is clear that the following three conditions are equivalent:

(i) the quotient of a presymplectic toric manifold
$(M^{2n+k},\omega, \rho)$ by the kernel isotropic foliation is 
Hausdorff;

(ii) all the leaves of the foliation are closed; 

(iii) the momentum polytope is rational. 

In the rational case, the leaves of the kernel isotropic foliation 
are $k$-dimensional tori (so we get a higher-dimensional
analogue of Seifert fibrations), and the quotient space 
$(M^{2n+k}/\text{kernel},\omega/\text{kernel}, \rho/\text{kernel})$ 
is a $2n$-dimensional symplectic toric orbifold. 

Compact symplectic toric orbifolds have been classified (up to 
equivariant symplectomorphisms) by Lerman and Tolman 
\cite{LeTo-Orbifold1997}: they put on each facet of
the momentum polytope a positive interger $m$, which corresponds to 
the orbifold type $D^{2(n-1)} \times D^2/\mathbb{Z}_m$ of points 
whose image under the momentum map lies in the facet. A rational 
convex polytope together with one positive interger for each facet 
is called a \textbf{\textit{weighted rational convex polytope}}. 
Lerman and Tolman \cite{LeTo-Orbifold1997} proved that
connected compact symplectic toric orbifolds are classified by their 
weighted rational convex polytopes (up to natural isomorphisms), 
and any weighted rational convex polytope can be realized 
by a compact symplectic toric orbifold. 

We can recover the above-mentioned result of Lerman and Tolman 
from our language of Morita-equivalent framed momentum polytopes 
as follows. 

Let $(P,Q)$ be a rational simple polytope with a regular
framing, $Q$ is of dimension $n+k$ and sits in $\mathbb{R}^{n+k}$,
$P = L \cap Q$ where $L$ is a rational $n$-dimensional affine subspace of
$\mathbb{R}^{n+k}$ which intersects $Q$ transversally. Let $\zeta_P$
be a facet of $P$ and $\zeta_Q$ be the corresponding facet of $Q$, 
$\zeta_P \subset \zeta_Q$. Fix a point $x \in \zeta_P$ and a basis
$\alpha_1,\hdots, \alpha_{n+k}$ of the integral lattice of $T_xQ$.
This basis can be chosen so that 
$\alpha_1,\hdots,\alpha_{n-1} \in T_x\zeta_P$ and
$\alpha_1,\hdots,\alpha_{n+k-1} \in T_x\zeta_Q$. The vector $\alpha_n$
does not belong to $T_xP$ in general, but there exists a linear 
combination $\beta = \sum_{i=1}^{n+k}c_i\alpha_i$ with integer 
coefficients $c_i$ such that $\beta \in T_x\zeta_P$ and $c_{n+k} > 0$. 
The minimal positive number $c_{n+k} > 0$ for which such an integral 
linear combination $\beta = \sum_{i=1}^{n+k}c_i\alpha_i \in T_x\zeta_P$ 
exists will be called the \textbf{\textit{weight}} of the facet 
$\zeta_P$ in the framing $(P,Q)$. It is easy to see that this number 
does not depend on the choice of the basis $(\alpha_1,\hdots, 
\alpha_{n+k})$. So each facet has a weight which is a positive 
integer, which depends only on the framed polytope.  

For any choice of weights for the facets of a given rational simple 
polytope, there always exists a regular framing with those weights. 
Indeed, in the construction of the cubic framing $Q$
given in Subsection \ref{subsection:LiftingFraming}, each facet of 
$Q$, which corresponds to a facet $\zeta_i$ of $P$, is given 
by an equation of the type
$$ y_i = \sum_{j=1}^n a_{ij} x_j + b_i,$$
(where $a_{ij}$ are integers because $P$ is rational),
and it is easy to check that the weight of this facet is nothing else 
but the greatest common divisor of the numbers $a_{i1}, \hdots, a_{in}$. 
By multiplying all the coefficients $a_{i1}, \hdots, a_{in}$ and 
$b_i$ by $p/q$, where $q$ is the greatest common divisor of 
$a_{i1}, \hdots, a_{in}$, and $p$ is any new weight that we want 
to have, we can change $Q$ to a new regular frame (while leaving 
$P$ unchanged), such that the weight of the facet $\zeta_i$ is 
changed from $q$ to $p$. 

If a regular framed polytope $(P,Q_1)$ admits a Morita equivalence 
embedding into a regular framed polytope $(P,Q_2)$, and $x$ is 
a point in a facet $\zeta_P$ of $P$, then a basis 
$(\alpha_1,\hdots, \alpha_{n+k})$ of the integral lattice of 
$T_xQ_1$ with the above properties can be completed to a basis 
of the integral lattice of $T_xQ_2$ with similar properties. This 
implies that the weight of $\zeta_P$ given by the framing $(P,Q_1)$ 
is the same as its weight given by the framing $(P,Q_2)$.
Hence facet weights are invariants with respect to Morita 
equivalence transformations of regular framed polytopes. 

Let us now show the converse: if two regular framings  
$Q_1$ and $Q_2$ of a rational simple convex polytope $P$ give 
rise to the same weight for each facet of $P$, then  $(P,Q_1)$ and 
$(P,Q_2)$ are Morita equivalent. In order to show it, we need
to construct a third regular framing $(P,Q_3)$ with Morita 
equivalence embeddings from $(P,Q_1)$ and $(P,Q_2)$ to $(P,Q_3)$. 
$Q_3$ can be constructed as the crossed
product of $Q_1$ and $Q_2$ relative to $P_3$ in a way which is 
absolutely similar to the construction in the proof of Theorem 
\ref{thm:MoritaPolytopes}. One then verifies directly that
$Q_3$ is regular, also in a similar way to the proof of Theorem 
\ref{thm:MoritaPolytopes}. So we obtain the following result, 
which incorporates the classification theorem of Lerman and Tolman 
\cite{LeTo-Orbifold1997}.

\begin{theorem} Consider two connected compact presympletic toric 
manifolds. The following conditions are equivalent:
\begin{itemize}
\item[{\rm (i)}] they are Morita equivalent;
\item[{\rm (ii)}] their quotients by the kernel isotropic foliations 
are isomorphic as symplectic toric orbifolds;
\item[{\rm (iii)}] their momentum polytopes are isomorphic and, 
moreover, have the same facet weights
given by the respective regular framings.
\end{itemize}
\end{theorem}

For irrational polytopes, we do not have orbifolds but quasifolds 
in the sense of  Battaglia--Prato \cite{BaPr_SimpleNonrational2001} 
(after a lifting and framing). In this case, our Morita equivalence 
for framed polytopes and presymplectic toric manifolds can be 
understood as a natural isomorphism relation among symplectic toric 
quasifolds. 

\section{Some final remarks} 
\label{section:Remarks}

\begin{remark}
{\rm
In this paper we considered only 
simple polytopes, but in fact any non-simple
convex polytope $P$ also admits a lifting and rational framing $(P',Q)$ 
by the same constructions. $Q$ still satisfies the rationality, 
simplicity and regularity conditions at its faces, $P'$ is still 
a slice of $Q$ by an affine subspace $L$. The only difference is 
that if $P'$ is not simple then $L$ intersects $Q$ 
\textit{non-transversally}. We still have a symplectic 
$2(n+k)$-dimensional manifold $(M^{2n+2k},\omega)$ with a 
Hamiltonian $\mathbb{T}^{n+k}$-action $\rho$ on it with a momnetum 
map $F$ such that $F(M) = Q$, and can still take $M_P = F^{-1}(P')$ 
to be the $(2n+k)$-dimensional presymplectic toric variety 
corresponding to the framed polytope $(P',Q)$. When $P$ is not 
simple then this presymplectic toric variety is singular (not a 
manifold) but still has very reasonable topology and geometry. 
When $P$ is rational non-simple then $P'$ is isomorphic to $P$, 
we can talk about a framing $(P,Q)$ of $P$, take the quotient of 
the singular presymplectic toric variety $M_P = F^{-1}(P')$ by 
the kernel isotropy foliation to get a \textit{singular symplectic 
toric variety} corresponding to $(P,Q)$ (with algebraic singularities). 
Some results about such singular symplectic toric varieties can be 
found, for example, in \cite{BuGuLe_ToricSingular2005}. 
}
\end{remark}

\begin{remark}
{\rm
In \cite{LeTo-Orbifold1997}, Lerman and Tolman extended the convexity 
theorem of Atiyah--Guillemin--Sternberg to the case of Hamiltonian 
torus actions on symplectic orbifolds. In this paper, we didn't study 
presymplectic orbifolds, but we are pretty sure that the presymplectic 
convexity theorem (Theorem \ref{thm:convexity}) can be naturally 
extended to the case of presymplectic orbifolds, with essentially the 
same arguments for the proof.
}
\end{remark}

\begin{remark}
{\rm
One can extend in a natural way the theory of \textit{symplectic cuts}
\cite{Lerman-Cut1995} to the presymplecyic setting, and to
presymplectic toric manifolds in particular. The corresponding 
operations on the level of framed momentum polytopes will also be 
cuts by rational hyperplanes. Some results concerning cuts for 
irrational polytopes and associated quasifolds were obtained recently 
by Battaglia and Prato \cite{BaPr_NonrationalCuts2016}. We recall 
again that their non-rational polytopes need to be lifted 
(non-isomorphically) before they can be framed and then cut.
}
\end{remark}

\begin{remark}
{\rm
In \cite{Battaglia_ConvexPolytope2012,BaZa_Irrational2015}, 
Battaglia and Zaffran also worked on
foliation and quotient modelings of irrational analogs 
of toric varieties. Their approach is complex-analytic, based 
on ideas from complex geometric invariant theory and earlier results 
of Meersseman and Verjovsky \cite{MeVe-Toric2004} and others; it is 
very different from our real presymplectic approach. 
In the case of rational polytopes, different approaches should
give basically the same results. 
}
\end{remark}

\begin{remark}
{\rm In this paper we didn't talk about Kähler structures at all, but
one can put compatible Kähler structures on symplectizations of 
presymplectic toric manifolds (these symplectizations are ``semi-local'' 
versions of symplectic toric manifolds), and use reduction 
(with respect to kernel torus actions) to get Kähler structures on 
quotient spaces, which are toric orbifolds or quasifolds or 
``non-commutative toric varieties'' in the language of 
\cite{KLMV-NoncommutativeToric2014}. 
}
\end{remark}

\begin{remark}
{\rm
We have the following \textit{global symplectization conjecture}, 
which is the global version of Theorem \ref{thm:LocalSymplectization},
and which looks very reasonable to us: \textit{With the assumptions of 
Theorem \ref{thm:LocalSymplectization}
there exists a connected compact symplectic manifold $\hat{M}$ with an
effective Hamiltonian $\mathbb{T}^{q+d}$-action such that $M$ is
a flat presymplectic cut of $\hat{M}$ with this action.} If one weakens
this conjecture and requires $\hat{M}$ to be an orbifold instead of a 
manifold, then it becomes rather easy.
}
\end{remark}

\section*{Acknowledgement}

This paper was written during  Nguyen Tien Zung's stay at the School of 
Mathematical Sciences, Shanghai Jiao Tong University, as a visiting 
professor. He would like to thank Shanghai Jiao Tong University, the 
colleagues at the School of Mathematics of this university, and 
especially Tudor Ratiu, Jianshu Li, and Jie Hu for the invitation, 
hospitality and excellent working conditions.

\end{document}